\documentclass[a4paper,11pt,leqno]{amsart}
\usepackage{amssymb}
\usepackage[alwaysadjust]{paralist}
\usepackage[capitalize]{cleveref}
\numberwithin{equation}{section}

\newtheorem{lem}[equation]{Lemma}
\newtheorem{prop}[equation]{Proposition}
\newtheorem{thm}[equation]{Theorem}
\theoremstyle{definition}
\newtheorem{defn}[equation]{Definition}

\newtheorem{rem}[equation]{Remark}

\def\IR{\mathbb R}

\def\eps{\varepsilon}

\newcommand{\id}{\operatorname{id}}

\newcommand{\supp}{\operatorname{supp}}

\newcommand{\vol}{\operatorname{vol}}
\newcommand{\Conf}{\operatorname{Conf}}
\newcommand{\capac}{\operatorname{Cap}}

\title{Regularity of conformal metrics with large first eigenvalue}
\author{Henrik Matthiesen}
\address
{Max Planck Institute for Mathematics,
Vivatsgasse 7, 53111 Bonn.}
\email{hematt\@@mpim-bonn.mpg.de}
\date{\today}

\begin{document}

\maketitle

\begin{abstract}
We establish a regularity result for conformal metrics with unit volume, $L^p$ scalar curvature
bounds for $p>n/2$ and first eigenvalue of $\Delta$ bounded from below by a constant $B > \Lambda_1(S^n,[g_{st.}]).$
\end{abstract}

\section{Introduction}

Let $(M,g)$ be a closed Riemannian manifold of dimension $n \geq 3.$
For a fixed metric $g$ we consider conformally related metrics of the form $u^{4/(n-2)}g,$ where $u$ is a smooth positive function.
The volumes of a measurable set $\Omega \subseteq M$ with respect to $g$ and $u^{4/(n-2)}g$ are related by
\begin{equation*}
\vol(\Omega,u^{4/(n-2)}g) = \int_\Omega u^{2^\star} dV_g,
\end{equation*} 
where $2^\star= 2n/(n-2)$ is the critical Sobolev exponent for the embedding
$W^{1,2}\hookrightarrow L^p$ and $V_g$ denotes the volume measure of $g.$

The scalar curvature transforms according to the semilinear elliptic equation
\begin{equation} \label{Yamabe}
4\frac{n-1}{n-2} \Delta_g u + R_g u = R_{u^{4/(n-2)}g} u^{2^\star -1}
\end{equation}
which is of critical nonlinearity.
Here $R_g$ and $R_{u^{4/(n-2)}g}$ denote the scalar curvature of the metric $g$ respectively $u^{4/(n-2)}g,$
and $\Delta_g=\Delta$ the (positive) Laplace operator of $(M,g).$

Thus one may view the scalar curvature as a 'Laplacian of the metric' when one considers
only a fixed conformal class.
In view of this analogy one may ask whether $L^p$-bounds on the scalar curvature imply $W^{2,p}$-bounds
on the conformal factors $u.$
In general, this is not true, see e.g.\ \cite{CGW}.
For $p>n/2,$ several results in this direction are known under certain additional assumptions.
Classical examples of such assumptions are that the first eigenvalue of the Laplace operator is bounded away
from $0$ and that additionally $(M,g)$ has dimension three \cite{BPY,CY2} or $(M,g)=(S^n,g_{st.})$ \cite{CY1}.
The assumptions on the geometry are made in order to rule out a blow-up at a single point.
In the case of the sphere, the result only holds true up to pulling back the conformal factors
by conformal transformations.
The large group of conformal diffeomorphisms of $S^n$ makes it possible to avoid blow-ups.
In dimension three a possible blow-up can be analyzed carefully and eventually ruled out.

We prove a result in the spirit of the results in \cite{BPY, CY1, CY2,Gursky}, but instead of geometric assumptions we assume that the
first eigenvalue is sufficiently large in order to rule out a possible blow-up.

Denote by $\omega_n$ the $n$-dimensional Euclidean volume of the unit ball.
Our main result is

 \begin{thm} \label{main}
 Let $(M,g)$ be a closed Riemannian manifold of dimension $n\geq 3$ and $u$ a smooth positive function.
 Consider the conformal metric $\tilde{g}=u^{4/(n-2)}g$ and denote by $\tilde{R}$ its scalar curvature.
Assume that
 \begin{itemize}
 	\item[(i)] $\vol(M,\tilde{g})=1$,
 	\item[(ii)] $\int_M |\tilde{R}|^p u^{2^\star} dV_g \leq A$ for some $n/2 < p< \infty,$
 	\item[(iii)] $\lambda_1(M,\tilde{g})\geq B > n\left( (n+1)\omega_{n+1} \right)^{2/n}.$ 
 \end{itemize}
 Then there exist constants $C_1,C_2, C_3>0$ depending on $(M,g)$ and $A,B$ such that
 $C_1 \leq u\leq C_2$ and $\|u\|_{W^{2,p}(M,g)}\leq C_3.$
 \end{thm}

Observe that 
once the $L^\infty$-bounds on $u$ and $u^{-1}$  are established,
the bound $\|u\|_{W^{2,p}(M,g)}\leq C_3$ is a consequence of the 
standard elliptic estimates in $L^p$-spaces applied to equation \eqref{Yamabe}, 
see e.g.\ \cite[Theorem 6.4.8]{Morrey_1966}

The geometric significance of the constant $B$ in assumption $(iii)$ is that we have $\lambda_1(S^n,g_{st.}) \vol(S^n,g_{st.})^{2/n}=n\left( (n+1)\omega_{n+1} \right)^{2/n},$
where $g_{st.}$ denotes the round metric on $S^n$ of curvature $1.$

\begin{rem}
Due to a result of Petrides \cite{Petrides}, any conformal class except for the standard conformal class on $S^n$ admits a smooth metric $\tilde g$ with unit volume and $\lambda_1(M,\tilde{g}) > n\left( (n+1)\omega_{n+1} \right)^{2/n}.$ 
See also \cite{CES}, where a related but weaker result is proved.
\end{rem}

\cref{main} has some immediate and interesting consequences.

As mentioned above, for $A$ sufficiently large we can find a constant $B>n\left( (n+1)\omega_{n+1} \right)^{2/n}$ such that there is at least one metric in the conformal class of $g$
 satisfying the assumptions of \cref{main} with these constants $A,B.$
For such $A,B$ it is possible to find a positive and H{\"o}lder continuous function $u,$
such that $u^{4/(n-2)}g$  maximizes $\lambda_1$ among all unit volume metrics in 
the conformal class of $g$
 satisfying the same $L^p$-bound on the scalar curvature, see \cref{conf_spec}

Another consequence is a compactness result for sets of isospectral metrics within a conformal class, which satisfy in addition the assumptions of \cref{main}, see \cref{isospec}

In \cref{proof_main} we explain the proof of \cref{main}.
Afterwards, in \cref{appl}, we briefly discuss the above mentioned applications.

\subsection*{Acknowledgment}
The author started working on this problem in his master's thesis at the University of Bonn and it is a great pleasure
to thank his advisor Werner Ballmann for suggesting this problem and many helpful discussions.
He is also very grateful to Bogdan Georgiev for many mathematical conversations.
The comments of the referee helped to improve the presentation significantly.
Moreover, the  support and hospitality of
the Max Planck Institute for Mathematics in Bonn are gratefully acknowledged.

\section{Proof of the theorem} \label{proof_main}

The main argument for the proof of \cref{main} is that assumption $(iii)$ rules
out a possible blow-up.
Once this is established, the result follows from arguments which are seen as fairly standard by now.
These arguments are based on the Moser iteration scheme.
For convenience, we explain how to lift the integrability in order to a find bound on $\|u\|_{2^\star+\eps}$ for some $\eps>0.$
This is proved in different ways in various places, but we could not locate a reference
stating precisely what we need.
Once this is done, the $L^\infty$-bounds on $u$ and $u^{-1}$ follow from a Harnack inequality established by Trudinger in \cite{Trudinger}.
Given the $L^\infty$-bounds the result follows from standard elliptic theory.

In general, all constants called $C$ may differ from line to line and will depend on $(M,g)$ and the data
$A,B.$

\subsection{A Volume non-concentration result}

We start with a few preparations.
From now on all metric quantities refer to the fixed background metric $g$
if not explicitly stated differently.
Recall the definition of $p$-capacities,
\begin{defn}
For a pair $(E,F)$ of  subsets $E \subset \subset \mathring{F} \subseteq M$ we define the \emph{$p$-capacity} by
\begin{equation*}
\capac_p(E,F):= \inf \int_M | d f |^p dV_g ,
\end{equation*}
where the infimum is taken over all Lipschitz functions $f \colon M \to \IR$ which are $1$ on $E$ and $0$ outside $F$.
\end{defn}
Note that since
\begin{equation} \label{invariance}
\int_M |d f |^n_{\tilde{g}} dV_{\tilde{g}}= \int_M |d f|^n_g dV_g,
\end{equation}
whenever $g$ and $\tilde{g}$ are conformally related, the $n$-capacity is conformally invariant.
We will use the following frequently.

\begin{lem} \label{n-capacity}
Let $R>0,$ then
for $p \leq n$ and any point $x \in M$, the $p$-capacity satisfies $\lim_{r \to 0} \capac_p(B(x,r),B(x,R))=0$.
\end{lem}

\begin{proof}
Observe that H{\"o}lder's inequality implies that
it suffices to consider the case $p=n.$
Moreover, it clearly suffices to prove the statement only for some small $R < inj(M).$ 
Let $0<r<R$ and define $\psi_{r,R}$ by
\begin{equation*}
\psi_{r,R}(z)= \begin{cases}\frac{\log(d_g(x,z)R^{-1})}{\log(rR^{-1})} & \text{ if } r \leq d_g(x,z) \leq R \\
							1 & \text{ if }  d_g(x,z) \leq r \\
							0 & \text{ if }  d_g(x,z) \geq R.
	\end{cases}
\end{equation*}
If $g$ is flat in $B(x,R),$ we have
\begin{equation*}
\int_{B(x,R)} | \nabla \psi_{r,R} |^n dV_g =\omega_n \left( \log \left( \frac{R}{r} \right) \right)^{1-n}.
\end{equation*}
In general, for $R>0$ such that $g$ is comparable on $B(x,R)$ to the Euclidean metric on $B(0,R),$ we have
\begin{equation*}
\int_{B(x,R)} | \nabla \psi_{r,R} |^n dV_g \leq C \left( \log \left( \frac{R}{r} \right) \right)^{1-n}.
\end{equation*}
We conclude
\[
\lim_{r \to 0} \int_{B(x,R)} | \nabla \psi_{r,R} |^n =0,
\]
thus $\lim_{r \to 0} \capac_n(B(x,r),B(x,R))=0.$ 
\end{proof}

Before we can prove the volume non-concentration result, we need the following observation
about conformal immersions which also appears in \cite{ElSoufi_Ilias}.

\begin{lem} \label{conformal}
Let $\Phi \colon (M,g) \to (S^n,g_{st.})$ be a conformal immersion.
Denote by $z^1, \dots, z^{n+1}$ the standard coordinate funtions of $\IR^{n+1}$ restricted to $S^n.$
Then 
\begin{equation} \label{conf_form}
\Phi^* g_{st.} = \frac{1}{n} \sum_{i=1}^{n+1} |\nabla(z^i \circ \Phi) |^2 g.
\end{equation}

\begin{proof}
First, observe that for $\gamma \in SO(n+1)$ we have
\begin{equation*} 
(\gamma \circ \Phi)^*g_{st.} =  \Phi^* g_{st.} 
\end{equation*}
and
\begin{equation*}
\frac{1}{n} \sum_{i=1}^{n+1} |\nabla(z^i \circ \Phi) |^2 g = \frac{1}{n} \sum_{i=1}^{n+1} |\nabla(z^i \circ \gamma \circ \Phi) |^2 g.
\end{equation*}
Thus it suffices to compute both sides of \eqref{conf_form} at a point $x \in M$ with $\Phi(x)=N,$ where $N=(0, \dots,0,1) \in S^n$ denotes the northpole.
In this case we may use the functions $z^1, \dots , z^n$ as coordinates about $N.$
We thus have coordinates $z^i \circ \Phi,$ $i=1, \dots, n$ about $x,$ such that $D \Phi(x)=\id$ in these coordinates.
Moreover, since $\Phi$ is conformal, there is positive constant $a$ such that $g_{jk}(x)=a \delta_{jk}.$
Thus we find 
\begin{equation*}
\frac{1}{n} \sum_{i=1}^{n+1} |\nabla(z^i \circ \Phi) |^2(x) g_{jk}(x)=\frac{a}{n} \sum_{i=1}^{n+1} \frac{1}{a} |\nabla z^i|^2(N)\delta_{jk}= \delta_{jk}=\left(\Phi^* g_{st.}\right)_{jk}(x).\\
\qedhere
\end{equation*}
\end{proof}

\end{lem}
The next proposition states that we can control the volume of small balls uniformly in terms
of the lower bound $B$ on the first eigenvalue.

\begin{prop} \label{non_conc_prop}
Let $u$ be a function satisfying assumptions $(i)$ and $(iii)$ in \cref{main}. 
For any $\delta>0,$ there is a radius $r=r(M,g,\delta,B)>0,$ such that $\int_{B(x,r)} u^{2^\star}dV_g<\delta$ for any $x \in M.$ 
\end{prop}

\begin{proof}
The idea of the proof is based on arguments due to Kokarev, who proved the same result in dimension two in
\cite{Kokarev}.
Kokarev used ideas developed by Nadirashvili in \cite{Nadirashvili_1996}.

Assume the statement is not correct.
Then we can find $\delta>0$ together with a sequence $x_k \in M$ of points and a sequence $u_k$ of smooth positive functions such that $\vol(M,u_k^{4/(n-2)}g)=1,$ $\lambda_1(M,u_k^{4/(n-2)}g)\geq B,$ and
$\int_{B(x_k,1/k)}u_k^{2^\star}dV_g \geq \delta.$
We denote by $g_k$ the metric $u_k^{4/(n-2)}g.$

Up to extracting a subsequence we can assume that the probability measures $V_{g_k}$ converge to a Radon probability measure $\mu$ in the weak*-topology.
Moreover, we may also assume that $x_k \to x.$
We claim that $\mu(\{x\})>0.$
In fact, 
take a sequence $\eta_l \in C_c^\infty (B(x,2/l))$ with $0 \leq \eta_l \leq 1$ and $\eta_l(x)=1$.
Then, by dominated convergence,
\begin{equation*}
\mu(\{ x\})= \lim_{l \to \infty} \int_M \eta_l d \mu.
\end{equation*}
So fix $\eta_l$ as above and assume in addition that $\eta_l =1$ on $B(x, 1/l).$
We thus have
\begin{equation*}
\begin{split}
\mu(\{x\}) &
= \lim_{l \to \infty} \lim_{k \to \infty} \int_{B(x,2/l)} \eta_l u_k^{2^\star} dV_g  
 \geq \lim_{l \to \infty} \lim_{k \to \infty} \int_{B(x,1/l)} u_k^{2^\star}dV_g \\
& \geq \lim_{l \to \infty} \lim_{k \to \infty} \int_{B(x_k,1/k)} u_k^{2^\star}dV_g 
 \geq \lim_{k \to \infty} \int_{B(x_k,1/k)} u_k^{2^\star}dV_g 
\geq \delta.
\end{split}
\end{equation*}

We consider two cases: The first case is $\mu=\delta_x$ for some $x \in M$.
In the remaining case we can find for a point $x \in M$ with $\mu(\{x\})>0$ a radius $R>0$ such that $\mu(M \setminus B(x,2R)) > 0.$
Let us start with the second case which is the easier one.

Take a ball $B(x,2R)$ as described above.
By \cref{n-capacity} we find $r>0$ such that $\capac_n(B(x,r),B(x,R))< \eps.$ 
Thus we can choose a Lipschitz function $\psi$ supported in $B(x,R),$ which satisfies $0 \leq \psi \leq 1,$ $\psi=1$ on $B(x,r)$ and
$\int_M | \nabla \psi |^n dV_g < \eps.$

Denote by $\alpha_k$ the mean (with respect to $g_k$) of $\psi$, i.e.
\begin{equation*}
\alpha_k= \int_M \psi dV_{g_k}.
\end{equation*} 
By the min-max principle, H{\"o}lder's inequality and the conformal invariance \eqref{invariance}, we find
\begin{equation*} 
\begin{split}
\lambda_1(g_k) \int_M(\psi-\alpha_k)^2 dV_{g_k} \leq & \int_M |d \psi |_{g_k}^2 dV_{g_k}  \\
  \leq & \left( \int_M |d \psi |_{g_k}^n dV_{g_k} \right)^{2/n} \left(\vol_{g_k} \left( \supp \psi \right)\right)^{(n-2)/n} \\
 =& \left( \int_M |\nabla \psi |^n dV_{g} \right)^{2/n} \left(\vol_{g_k} \left( \supp \psi \right)\right)^{(n-2)/n} \\
\leq & \eps^{2/n}.
\end{split}
\end{equation*}
We can estimate the left-hand-side from below by $\lambda_1(g_k)$ times
\begin{equation*}
\alpha_{k}^2\int_{M \setminus B(x,R)} u_k^{2^\star} dV_g + (1-\alpha_{k})^2 \int_{B(x,r)}  u_k^{2^\star} dV_g.
\end{equation*}
Let us investigate both terms as $k \to \infty$.
We have
\begin{equation*}
\liminf_{k \to \infty} \int_{M \setminus B(x,R)} u_k^{2^\star} dV_g \geq \mu(M \setminus B(x,2R)),
\end{equation*}
and
\begin{equation*}
\liminf_{k \to \infty} \int_{B(x,r)} u_k^{2^\star} dV_g \geq  \mu(\{x\}).
\end{equation*}
Up to extracting a subsequence we may assume 
$\alpha_{k} \to \alpha$ as $k \to \infty$.
By construction $\alpha \in [0,1]$, thus we find
\begin{equation*}
\begin{split}
\limsup_{k \to \infty} \lambda_1(g_k) & \leq \frac{\eps^{2/n}}{\alpha^2\mu(M \setminus B(x,2R))+(1-\alpha)^2\mu(\{x\})}\\
& \leq \frac{4 \eps^{2/n}}{\min\{\mu(M \setminus B(x,2R)),\mu(\{x\}) \}}.
\end{split}
\end{equation*}
And thus $\limsup_{k \to \infty} \lambda_1(g_k)=0$, a contradiction.

The case $\mu = \delta_x$ is slightly more involved.
In a first step we observe that we may assume without loss of generality that $g$ is flat near $x.$
This observation is motivated by the arguments in \cite{CES}.

Given any $\eps>0$ we can replace $g$ by a another metric $g'$ which is flat near
$x$ and $(1+\eps)$-quasiisometric to $g,$  
i.e.\ $(1+\eps)^{-2} g(v,v) \leq g'(v,v) \leq (1+\eps)^2 g(v,v)$
for all non-zero tangent vectors $v,$
see e.g.\ \cite[Lemma 2.3]{CES}.
Then for each $k$ the metric $u_k^{4/(n-2)} g'$ is $(1+\eps)$-quasiisometric to the metric $u_k^{4/(n-2)} g.$
Rescale the functions $u_k$ to obtain new functions $u'_k$ such that we have $\vol(M,\left(u'_k\right)^{4/(n-2)} g')=1.$
Since the volumes of $(M,u_k^{4/(n-2)}g)$ and $(M, u_k^{4/(n-2)} g')$ are controlled by
\begin{equation} \label{volume}
(1+\eps)^{-n}  \leq \frac{\vol(M,u_k^{4/(n-2)}g')}{\vol(M,u_k^{4/(n-2)}g)} \leq (1+\eps)^n ,
\end{equation}
we find that 
the ratios $u_k^{4/(n-2)}/(u'_k)^{4/(n-2)}$ are uniformly bounded from above and below by 
$(1 + \eps)^2$ respectively $(1 + \eps)^{-2}.$
This implies that $u_k^{4/(n-2)} g$ is $(1+\eps)^2$-quasiisometric to $(u')_k^{4/(n-2)} g',$
thus
\begin{equation*}
\lambda_1\left(\left(u'_k\right)^{4/(n-2)} g'\right) \geq
(1+\eps)^{-4(n+1)} \lambda_1\left(u_k^{4/(n-2)} g\right).  
\end{equation*}
In particular, for $\eps$ sufficiently small we find $B'>n\left( (n+1)\omega_{n+1} \right)^{2/n},$ such that
$\lambda_1(M,\left(u'_k\right)^{4/(n-2)} g')\geq B'.$

Similarly as in \eqref{volume}, we have for any measurable subset $\Omega \subseteq M$ that
\begin{equation*} \label{volume2}
\vol(\Omega,(u'_k)^{4/(n-2)}g') \leq (1+\eps)^{2n} \vol(\Omega,u_k^{4/(n-2)}g),
\end{equation*}
since  $u_k^{4/(n-2)} g$ is $(1+\eps)^2$-quasiisometric to $(u'_k)^{4/(n-2)} g'.$ 
Applying this to subsets of $M \setminus \{x\}$ easily implies that $(u'_k)^{2^\star} V_{g'} \rightharpoonup^* \delta_x. $ 
In more detail, if $\nu$ is the weak*-limit of a subsequence of $(u'_k)^{2^\star} V_{g'},$
we have for any open $\Omega \subset \subset M \setminus \{x\}$ that
\begin{equation*}
\begin{split}
\nu(\Omega) 
& \leq \liminf_{k \to \infty} \vol (\Omega , (u'_k)^{4/(n-2)}g' ) 
\\
& \leq \liminf_{k \to \infty} (1+\eps)^{2n} \vol (\Omega , (u_k)^{4/(n-2)}g) 
\\
& \leq \limsup_{k \to \infty} (1+\eps)^{2n} \vol (\overline{\Omega}, (u_k)^{4/(n-2)}g)
\\
&\leq (1+\eps)^{2n} \mu (\overline{\Omega})
\\
&=0,
\end{split}
\end{equation*}
since $u_k^{2^\star}V_g \rightharpoonup^* \mu= \delta_x.$
This implies $\nu(M \setminus \{x\})=0$ and thus $\nu=\delta_x.$
We have shown that the limit of any weakly*-convergent subsequence of $(u'_k)^{2^\star} V_{g'}$ has to be $\delta_x.$
Since every subsequence of $(u'_k)^{2^\star} V_{g'}$ has a weakly*-convergent subsequence, it follows that $(u'_k)^{2^\star} V_{g'} \rightharpoonup^*  \delta_x.$  
Now it suffices to show that such a sequence $(u_k')$ can not exist on $(M,g').$
  
Let $\Omega$ be a conformally flat neighborhood of $x.$
Choose a conformal immersion $\Phi \colon (\Omega , g) \to (S^n,g_{st.})$.
By diminishing $\Omega$ if necessary we may assume that $\Phi$ is an embedding. 
Fix $\eps >0$ and choose a function $\psi \in W^{1,\infty}_0(\Omega)$  with $\int_\Omega | \nabla \psi |^n dV < \eps$, $0 \leq \psi \leq 1$, and $\psi(x)=1$, which is possible thanks
to \cref{n-capacity}.
By Lemma \ref{Hersch} below, we find $s_k \in \Conf(S^n)$ such that
\begin{equation*}
\int_\Omega \psi \cdot (z^i \circ s_k \circ \Phi) dV_{g_k}=0
\end{equation*}
for all $i=1, \dots , n$.
Using the functions $\psi \cdot (z^i \circ s_k \circ \Phi)$ as test functions and summing over all $i$ yields
\begin{equation}  \label{sum1}
\begin{split}
\lambda_1(g_k) \int_\Omega \psi^2 dV_{g_k} 
=&  \lambda_1(g_k) \sum_{i=1}^{n+1} \int_\Omega (\psi z^i)^2 dV_{g_k} \\
 \leq& \sum_{i=1}^{n+1} \int_\Omega |d (\psi \cdot (z^i \circ s_k \circ \Phi))|^2_{g_k} dV_{g_k} \\
 = &\sum_{i=1}^{n+1} \int_\Omega |d \psi|_{g_k}^2 (z^i \circ s_k \circ \Phi)^2 dV_{g_k} \\
& + 2 \sum_{i=1}^{n+1} \int_\Omega (z^i \circ s_k \circ \Phi) \psi \langle d \psi, d(z^i \circ s_k \circ \Phi) \rangle_{g_k} dV_{g_k}  \\
& + \sum_{i=1}^{n+1} \int_\Omega \psi^2 |d(z^i \circ s_k \circ \Phi)|^2_{g_k} dV_{g_k}.
\end{split}
\end{equation}
By H\"older's inequality and conformal invariance, the first summand in \eqref{sum1} can be controlled as follows
\begin{equation} \label{summand1}
\begin{split} 
\sum_{i=1}^{n+1} \int_\Omega |d \psi|_{g_k}^2 (z^i \circ s_k \circ \Phi)^2 dV_{g_k}
= & \int_\Omega |d \psi|_{g_k}^2 dV_{g_k}  \\
\leq & \left( \int_\Omega |d \psi|_{g_k}^n dV_{g_k} \right)^{2/n} \vol_{g_k}(\Omega)^{(n-2)/n}  \\
\leq & \left( \int_\Omega |\nabla \psi|dV_{g} \right)^{2/n} \leq \eps^{2/n}.
\end{split}
\end{equation}
For the second summand in \eqref{sum1} notice that
\begin{equation} \label{summand2_1}
\begin{split}
\int_\Omega | d (z^i\circ s_k \circ \Phi) |_{g_k}^n dV_{g_k} 
&=
\int_{s_k \circ \Phi(\Omega)} | \nabla z^i |^n dV_{g_{st.}} \\
& \leq  C \vol(s_k \circ \Phi(\Omega)) 
\leq C,
\end{split}
\end{equation}
for a constant $C=C(n),$ thanks to conformal invariance.
This implies
\begin{equation} \label{summand2_2}
\begin{split}
\sum_{i=1}^{n+1}& \int_\Omega (z^i \circ s_k \circ \Phi) \psi \langle d \psi, d(z^i \circ s_k \circ \Phi) \rangle_{g_k} dV_{g_k} \\
\leq & 
 \sum_{i=1}^{n+1} \sup_{x \in \Omega}|(z^i \circ s_k \circ \Phi)| \int_\Omega |d \psi|_{g_k} |d (z^i \circ s_k \circ \Phi)|_{g_k} dV_{g_k}   \\
 \leq &
 \sum_{i=1}^{n+1}  \left( \int_\Omega |d \psi|^n_{g_k}  dV_{g_k}\right)^{1/n} \left( \int_\Omega |d (z^i \circ s_k \circ \Phi) |^n_{g_k} dV_{g_k} \right)^{1/n} \vol_{g_k}(\Omega)^{(n-2)/n} \\
 \leq &
 C \eps^{1/n}.
\end{split}
\end{equation}
The last summand in \eqref{sum1} is estimated using \cref{conformal},
\begin{equation} \label{summand3}
	\begin{split}
			\sum_{i=1}^{n+1} \int_\Omega \psi^2 |d(z^i \circ s_k \circ \Phi)|^2_{g_k} dV_{g_k}
		\leq & 
			\left( \int_\Omega \left( \sum_{i=1}^{n+1} |d(z^i \circ s_k \circ \Phi)|^2_{g_k} \right)^{n/2} dV_{g_k} \right)^{2/n} \\
		= & 
			n \vol ((s_k \circ \Phi)(\Omega))^{2/n} 
		\leq 
			n((n+1)\omega_{n+1})^{2/n}.
	\end{split}
\end{equation}
Combining \eqref{sum1}, \eqref{summand1}, \eqref{summand2_2} and \eqref{summand3}, we conclude
\begin{equation*}
\begin{split}
\limsup_{k \to \infty} \lambda_1(u_k^{4/(n-2)} g) 
= &
\limsup_{k \to \infty} \lambda_1(u_k^{4/(n-2)} g) \int_\Omega \psi dV_{g_k} \\
\leq & 
 \eps^{2/n} + C \eps^{1/n}  + n \left( (n+1) \omega_{n+1}\right)^{2/n},
 \end{split}
\end{equation*}
which proves our claim.
\end{proof}

Next, we give the version of the Hersch lemma, which we have used in the proof of \cref{non_conc_prop}.
\begin{lem}[Hersch lemma] \label{Hersch}
Let $\mu$ be a continuous Radon measure on $M$, $\psi \in W_0^{1,\infty}(\Omega),$
where $\Omega \subseteq M$.
Moreover, assume $0 \leq \psi \leq 1.$
Then for any fixed conformal map $\Phi \colon (\Omega,g) \to (S^n,g_{st.})$ there exists a conformal diffeomorphsim
$s \in \Conf(S^n)$ such that
\begin{equation*}
\int_\Omega \psi \cdot (z^i \circ s \circ \Phi) d\mu =0,
\end{equation*}
for all $i=1,\dots,n+1.$
\end{lem}

This can be proved in complete analogy to the original Hersch lemma, see \cite{Hersch}.
For convenience of the reader we recall the main idea of the elegant argument.

Denote by $s_x$ the stereographic projection onto $T_xS^n.$
For $x \in S^n$ and $\lambda \in \IR_{>0}$ we have conformal diffeomorphisms $g_{x,\lambda}$  of $S^n$
given by $g_{x,\lambda}(y)=s_x^{-1}(\lambda s_x(y))$ for $y \neq -x$ and $g_{x,\lambda}(-x)=-x.$
Consider the map $C \colon D^{n+1} \to D^{n+1}$ given by 
\begin{equation*}
C(\lambda x) = \left(\int_\Omega \psi d\mu \right)^{-1} \int_\Omega \psi \cdot (z \circ g_{x,1-\lambda} \circ \Phi) d\mu,
\end{equation*} 
where $\lambda \in [0,1),$ and $x \in S^n.$
It is easily checked that this extends to a continous map $\overline{C} \colon \overline{D}^{n+1} \to \overline{D}^{n+1},$ which
restricts to the identity on $S^{n}.$
It follows, that there have to be $x, \lambda$ such that $C(\lambda x)=0.$

\subsection{Higher integrability of the conformal factors}

The next step is to improve the integrability of the conformal factors.
Once this is done a Harnack inequality due to Trudinger \cite{Trudinger} gives $L^\infty$-bounds and  \cref{main} follows from 
the standard elliptic estimates.

\begin{lem} \label{higher_int_lem}
For $u$ as in \cref{main}, 
there are $D,\eps>0$ depending on $(M,g)$ and $A,B$ such that 
\begin{equation*} 
\int_M u^{2^\star +\eps} dV_g \leq D. 
\end{equation*}
\end{lem}

\begin{proof}
We follow the proof of Lemma 2.3 in \cite{Gursky}.

Of course, it suffices to prove that there are constants $\eps=\eps(n)>0,$ $r=r(M,g,A,B)>0,$ and $C=C(M,g,A,B)$ such that
$\int_{B(x,r)} u^{2^\star+\eps}dV_g \leq C$
for any $x \in M.$

For $r>0$ to be chosen later we take a smooth function $\eta  \colon M \to [0,1]$ with $\eta=1$ on $B(x,r/2)$, $\eta=0$ outside $B(x,r),$
and $| \nabla \eta | \leq C/r.$ 
If we choose  $\eps>0$ such that $2+2\eps \leq 2^\star,$ we get from the Sobolev inequality that
\begin{eqnarray}  \label{sobolev}
\begin{split}
\left(\int_M \left(\eta u^{1+\eps}\right)^{2^\star} dV_g \right)^{2/2^\star} 
& \leq C \left( \int_M |\nabla (\eta u^{1+\eps})|^2 dV_g + \int_M \eta^2 u^{2+2\eps} dV_g\right)\\
 & \leq C \int_M |\nabla (\eta u^{1+\eps})|^2 dV_g+ C \left( \int_M u^{2^\star}dV_g\right)^{(2+2\eps)/2^\star} \\ 
& \leq C \int_M |\nabla (\eta u^{1+\eps})|^2 dV_g + C, 
\end{split}
\end{eqnarray} 
using H{\"o}lder's inequality and the volume bound.

In order to estimate the first summand above we multiply equation \eqref{Yamabe} by $\eta^2 u^{1+2\eps}$ and integrate by parts in order to find 
\begin{equation*} 
\begin{split}
 (1+\eps) \int_M \eta^2  u^{2\eps} |\nabla u|^2 dV_g & 
 =\frac{n-2}{4(n-1)}  \left( \int_M \tilde{R} \eta^2 u^{2^\star +2\eps}dV_g
- \int_M R \eta^2 u^{2+2\eps}dV_g  \right)\\
& -2  \int_M  \eta u^{1+2 \eps} \nabla \eta \nabla u dV_g
- \eps \int_M \eta^2  u^{2\eps} |\nabla u|^2dV_g.
\end{split}
\end{equation*} 
Inserting this into
\begin{equation*} 
\begin{split}
\int_M |\nabla (\eta u^{1+\eps})|^2 dV_g
 = & \int_M u^{2+2\eps} |\nabla \eta|^2dV_g
+ 2 (1+\eps) \int_M \eta u^{1+2\eps}  \nabla \eta \nabla u dV_g\\
& +  (1+\eps)^2 \int_M \eta^2 u^{2 \eps} |\nabla u|^2 dV_g
\end{split}
\end{equation*}
gives
\begin{equation} \label{iter}
\begin{split}
& \int_M |\nabla (\eta u^{1+\eps})|^2 dV_g
=  \int_Mu^{2+2\eps} |\nabla \eta|^2 dV_g
- (1+\eps) \eps  \int_M \eta^2  u^{2\eps} |\nabla u|^2 dV_g \\
&  \ \ \ + (1+\eps)\frac{n-2}{4(n-1)}\left( 
\int_M\tilde{R} \eta^2 u^{2^\star +2\eps}dV_g
-\int_M R \eta^2 u^{2+2\eps}dV_g
\right) \\
& \leq
\int_Mu^{2+2\eps} |\nabla \eta|^2 dV_g
 + C\left( 
\int_M\tilde{R} \eta^2 u^{2^\star +2\eps}dV_g
-\int_M R \eta^2 u^{2+2\eps}dV_g
\right).
\end{split}
\end{equation}
Let us discuss all summands above separately.
By the choice of $\eta$ and the volume bound, we have, using H{\"o}lder's inequality,
\begin{equation} \label{summand_1}
\int_Mu^{2+2\eps} |\nabla \eta|^2 dV_g \leq \frac{C}{r^2}.
\end{equation}
By  H{\"o}lder's inequality and the volume bound again, we find
\begin{equation*}
\int_M\tilde{R} \eta^2 u^{2^\star +2\eps}dV_g 
 \leq \left(  \int_M |\tilde{R}|^p u^{2^\star} dV_g\right)^{1/p} \left( \int_M  (\eta u^\eps)^{2p/(p-1)} u^{2^\star} dV_g\right)^{(p-1)/p}.
\end{equation*}
Observe that $p>n/2$ implies
$q=n(p-1)/(p(n-2))>1.$
Thus by H{\"o}lder's inequality
\begin{equation*} \label{holder}
\int_M  (\eta u^\eps)^{2p/(p-1)} u^{2^\star} dV_g \leq \left( \int_M \left(\eta u^{1+\eps}\right)^{2^\star} dV_g\right)^\frac{1}{q} \left(\int_{B(x,r)} u^{2^\star}dV_g\right)^{(q-1)/q},
\end{equation*}
which implies
\begin{equation*}
\int_M\tilde{R} \eta^2 u^{2^\star +2\eps}dV_g 
 \leq A^{1/p} \left(\int_M \left(\eta u^{1+\eps}\right)^{2^\star} dV_g \right)^{2/2^\star} \left(\int_{B(x,r)} u^{2^\star}dV_g\right)^{(2p-n)/np}.
\end{equation*}
Since $p>n/2$ and using \cref{non_conc_prop}, we find $r=r(M,g,A,B),$ such that
\begin{equation*}
\left(\int_{B(x,r)} u^{2^\star}dV_g\right)^{(2p-n)/np} \leq \frac{1}{2A^{1/p}C}.
\end{equation*}
For such $r$ we conclude that
\begin{equation} \label{summand_2}
C \int_M\tilde{R} \eta^2 u^{2^\star +2\eps}dV_g 
\leq  \frac{1}{2} \left(\int_M \left(\eta u^{1+\eps}\right)^{2^\star} dV_g \right)^{2/2^\star}.
\end{equation}
The last summand is controlled by the volume bound
\begin{equation} \label{summand_3}
\int_M R \eta^2 u^{2+2\eps}dV_g \leq C \int_M u^{2+2\eps}dV_g \leq C.
\end{equation}
Combining \eqref{sobolev}-\eqref{summand_3} we conclude
\begin{equation*}
\left(\int_M \left(\eta u^{1+\eps}dV_g\right)^{2^\star} \right)^{2/2^\star}
\leq C+ \frac{C}{r^2} + \frac{1}{2} \left(\int_M \left(\eta u^{1+\eps}\right)^{2^\star} dV_g \right)^{2/2^\star},
\end{equation*}
and thus
\begin{equation*}
\int_{B(x,r/2)} u^{2^\star(1+\eps)}dV_g \leq C,
\end{equation*}
with $\eps=\eps(n),$ $r=r(M,g,A,B),$ and $C=C(M,g,A,B,r).$
\end{proof}

In order to prove \cref{main} we need the following Harnack inequality,
which can be found in \cite{Trudinger}.

\begin{lem} \label{harnack} 
Let $u \in W^{1,2}(M,g)$ be a non-negative solution of the elliptic equation
 $\Delta u = f u, $
 with $f \in L^q(M,g)$ for some $q>n/2.$
 Then there is $C=C(M,g,\|u\|_{L^2(M,g)},\|f\|_{L^q(M,g)})$ such that
\begin{equation*}
 C^{-1} \leq u \leq C.
\end{equation*}  
\end{lem}

We now turn to the proof of \cref{main}. 

\begin{proof}[Proof of \cref{main}]
For $n/2 <q <p,$ we have
\begin{equation*}
\begin{split}
\int_M & |\tilde{R}|^q u^{(2^\star-2)q} dV_g
 = \int_M |\tilde{R}|^q u^{(2^\star-2)q-2^\star} u^{2^\star} dV_g \\
& \leq \left( \int_M |\tilde{R}|^p u^{2^\star} dV_g \right)^{q/p} \left(\int_M u^{((2^\star-2)q-2^\star)p/(p-q)+2^\star} dV_g \right)^{(p-q)/p}.
\end{split}
\end{equation*}
Observe that $((2^\star-2)q-2^\star)p/(p-q)+2^\star \to 2^\star,$ as $q \to n/2.$
Thus we can find $q>n/2,$ such that $((2^\star-2)q-2^\star)p/(p-q)+2^\star \leq 2^\star + \eps,$ for
$\eps$ given by \cref{higher_int_lem}.
For such $q$ we have
\begin{equation*}
\int_M  |\tilde{R} u^{(2^\star-2)}|^q dV_g \leq D^{(p-q)/p} A^{q/p}.
\end{equation*}
In particular, we can apply 
 \cref{harnack} to 
\begin{equation} \label{yamabe_1}
 4 \frac{n-1}{n-2} \Delta u = \left( \tilde{R} u^{2^\star-2} - R\right) u
\end{equation}
and conclude that
we have $C_1 \leq |u| \leq C_2$ with $C_1,C_2$ depending on $M,g,A,B,D.$
This in turn implies that the right hand side of \eqref{yamabe_1} is bounded in $L^p(M,g).$
Thus the standard elliptic estimates \cite[Theorem 6.4.8]{Morrey_1966} in $L^p$-spaces imply
that $\|u\|_{2,p} \leq C_3.$
\end{proof}

\section{Applications} \label{appl}

We discuss two applications of \cref{main}, one of which was in fact a motivation for studying
this problem.

\subsection{Conformal spectrum}
Thanks to the work of Li--Yau \cite{Li_Yau}, El Soufi--Ilias \cite{ElSoufi_Ilias} and Korevaar \cite{Korevaar} the scale invariant quantities 
$\lambda_k(M,g) \vol(M,g) ^{2/n}$ are bounded within a fixed conformal class.
Thus it is a natural question to ask whether there are metrics realizing
$\sup_{\phi}\lambda_1(M,\phi g) \vol(M, \phi g) ^{2/n},$
where the supremum is taken over all smooth positive functions $\phi.$
In dimension two the conformal covariance of the Laplace operator simplifies the situation tremendously, 
but it remains a very difficult problem which was resolved only recently (see \cite{Kokarev, Petrides_2}.)
Also, it follows from the appendix in \cite{CY2} that there are H{\"o}lder continuous maximizers in dimension two, if one additionally imposes $L^p$ curvature bounds for $p > 1.$
\cref{main} generalizes this partially to higher dimensions.

For $p>n/2,$ $A>0$ and $B >n((n+1)\omega_{n+1})^{2/n},$ denote by $[g]_{A,B}$ the subset of the conformal class $[g]$ consisting of all metrics of the form $u^{4/{n-2}}g,$ such that $u \in W^{2,p}$ with
$\vol(M, u^{4/{n-2}}g)=1,$ $\int_M |R_{ u^{4/{(n-2)}}g}|^p u^{2^\star} dV_g \leq A $ and
$\lambda_1(M, u^{4/{n-2}}g) \geq B.$
Thanks to \cite{Petrides}, there are $A,B$ as above such that $[g]_{A,B}$ is non-empty. 

We have
\begin{thm} \label{conf_spec}
Let $p>n/2,$ $A>0$ and $B >n((n+1)\omega_{n+1})^{2/n},$ such that $[g]_{A,B} \neq \emptyset.$
Then there is a H{\"o}lder continuous positive function $u,$ such that $u^{4 / (n-2)}g \in [g]_{A,B}$ and
 $\lambda_1 (M, u^{4/(n-2)}g) = \sup_{h \in [g]_{A,B}} \lambda_1(h).$
\end{thm}

\begin{proof}
Let $(u_k)$ be a sequence of functions such that
$u_k^{4/(n-2)}g \in [g]_{A,B}$ and
$\lim_{k \to \infty} \lambda_1(M, u_k^{4/(n-2)}g) = \sup_{h \in [g]_{A,B}} \lambda_1(h). $
Due to \cref{main}, $(u_k)$ is bounded in $W^{2,p}.$
Thus we have a subsequence (not relabeled) $u_k \rightharpoonup u_*$ in $W^{2,p}.$
By the standard embedding results for Sobolev spaces we have that $W^{2,p} \hookrightarrow C^{0, \alpha}$ for some
$\alpha > 0,$ since $p > n/2.$
Thanks to the Theorem of Arzela--Ascoli, the embedding $C^{0, \alpha} \hookrightarrow C^{0, \beta}$ is compact for $\beta < \alpha.$
Thus we can extract a further subsequence (again not relabeled) such that
$u_k \to u_*$ in $C^{0, \beta}.$
Since the functional $\lambda_1$ is continuous with respect to convergence in $C^0,$
we see that $\lambda_1(M,u_*^{4/(n-2)}g) = \sup_{h \in [g]_{A,B}} \lambda_1(h).$

It remains to prove that $u_*^{4/(n-2)}g \in [g]_{A,B}.$
Since we have uniform upper and lower bounds for $u_k$ and thus also for $u_*,$ we can write thanks to
\eqref{Yamabe} that
\begin{equation}
 R_k= \frac{4(n-1)/(n-2)\Delta u_k + R u_k}{u_k^{2^\star-1}}
\end{equation}
and similarly for $R_*.$
Since $\Delta u_k \rightharpoonup \Delta u_*$ in $L^p(M,g)$ and $u_k \to u_*$ in $C^0,$ this implies
that $R_k \rightharpoonup R_*$ in $L^p(M,g).$
Moreover, by the uniform upper and lower bounds on $u_*,$ this implies that 
$R_k \rightharpoonup R_*$ in $L^p(M,u_*^{4/(n-2)}g).$
This implies
\begin{equation}
\int_M |R_*|^p u^{2^\star} dV_g \leq
\liminf_{k \to \infty} \int_M |R_k|^p u^{2^\star} dV_g
= \liminf_{k \to \infty} \int_M |R_k|^p u_k^{2^\star} dV_g
\end{equation}
and thus $u_*^{4/(n-2)}g \in [g]_{A,B}.$
\end{proof}

\subsection{Isospectral metrics}
For a fixed metric $g$ denote by  $\mathcal{I}(g)$ the set of all metrics on $M$ isospectral
to $g.$
With the same arguments as in in the proof of \cref{conf_spec} we find 
\begin{thm} \label{isospec}
Let $(M,g)$ and $A>0$ be such that $\vol(M,g)=1,$ $\lambda_1(M,g)> n((n+1)\omega_{n+1})^{2/n}$ and $g \in [g]_{A,\lambda_1(M,g)}.$
Then the set $\mathcal{I}(g) \cap [g]_{A,\lambda_1(M,g)}$ is precompact in $C^{0,\alpha}$ for some $\alpha>0.$
\end{thm}


\begin{thebibliography}{HMZ2}

\bibitem{BPY}
R.\, Brooks, P.\, Perry, P.\, Yang,
Isospectral sets of conformally equivalent metrics.
\emph{Duke Math. J.} {\bf 58} (1989), no.1, 131--150.

\bibitem{CGW}
S.-Y.\ A.\ Chang, M.\ Gursky, T.\ Wolff,
Lack of compactness in conformal metrics with {$L^{d/2}$} curvature.
\emph{J. Geom. Anal.} {\bf 4} (1994), no.2, 143--153.

\bibitem{CY1}
S.-Y.\ A.\ Chang, P.\ C.-P.\ Yang,
Compactness of isospectral conformal metrics on {$S^3$}.
\emph{Comment. Math. Helv.} {\bf 64} (1989), no.3, 363--374.


\bibitem{CY2}
S.-Y.\ A.\ Chang, P.\ C.-P.\ Yang,
Isospectral conformal metrics on {$3$}-manifolds.
\emph{J. Amer. Math. Soc.} {\bf 3} (1990), no.1, 117--145.


\bibitem{CES}
B.\, Colbois, A.\, El Soufi,
Extremal eigenvalues of the {L}aplacian in a conformal class of metrics: the `conformal spectrum'.
\emph{Ann. Global Anal. Geom.} {\bf 24} (2003), no.4 337--349.


\bibitem{ElSoufi_Ilias}
A.\, EL Soufi, S.\, Ilias,
Immersions minimales, premi\`ere valeur propre du laplacien et volume conforme.
\emph{Math. Ann.} {\bf 275}, no.2 (1986), 257--267. 


\bibitem{Gursky}
M.\, Gursky, 
Compactness of conformal metrics with integral bounds on curvature.
\emph{Duke Math. J.} {\bf 72} (1993), no.2, 339--367.

\bibitem{Hersch}
J.\ Hersch,
Quatre propri\'et\'es isop\'erim\'etriques de membranes sph\'eriques homog\`enes.
\emph{C. R. Acad. Sci. Paris S\'er. A-B} {\bf 270} (1970), A1645--A1648


\bibitem{Kokarev}
G.\, Kokarev,
Variational aspects of {L}aplace eigenvalues on {R}iemannian surfaces.
\emph{Adv. Math.} {\bf 258} (2014), 191--239.

\bibitem{Korevaar}
N.\, Korevaar,
Upper bounds for eigenvalues of conformal metrics.
\emph{J. Differ. Geom.} {\bf 37}, no.1 (1993), 73--93.

\bibitem{Li_Yau}
P.\, Li, S.\, T.\, Yau,
A new conformal invariant and its applications to the
              {W}illmore conjecture and the first eigenvalue of compact
              surfaces.
\emph{Invent. Math} {\bf 69}, no.2 (1982), 269--291.

\bibitem{Morrey_1966}
C.\, B.\, Morrey\, Jr.,
Multiple integrals in the calculus of variations.
\emph{Die Grundlehren der mathematischen Wissenschaften} {Band \bf 130} (1966), ix+506 pp. 

\bibitem{Nadirashvili_1996}
N.\, Nadirashvili,
{B}erger's isoperimetric problem and minimal immersions of
              surfaces.
\emph{Geom. Funct. Anal.} {\bf 6}, no.5 (1996), 877--897. 


\bibitem{Petrides}
R.\ Petrides,
On a rigidity result for the first conformal eigenvalue of the {L}aplacian.
\emph{J. Spectr. Theory} {\bf 5} (2015), no.1 227--234.

\bibitem{Petrides_2}
R.\ Petrides,
Existence and regularity of maximal metrics for the first Laplace eigenvalue on surfaces.
\emph{Geom. Funct. Anal.} {\bf 4} (2014), no.4 1336--1376.


\bibitem{Trudinger}
N.\ S.\ Trudinger,
Remarks concerning the conformal deformation of {R}iemannian structures on compact manifolds.
\emph{Ann. Scuola Norm. Sup. Pisa (3)} {\bf 22} (1968), 265--274.





\end{thebibliography}
\end{document}